\documentclass[12pt]{amsart}
\usepackage{ amsmath, amsthm, amsfonts, amssymb, color}
 \usepackage{mathrsfs}
\usepackage{amsfonts, amsmath}
 \usepackage{amsmath,amstext,amsthm,amssymb,amsxtra}
 \usepackage{txfonts} 
 \usepackage[colorlinks, citecolor=blue,pagebackref,hypertexnames=false]{hyperref}
 \allowdisplaybreaks
 \usepackage{multirow}
 \usepackage{diagbox} 

\usepackage{mathtools} \mathtoolsset{showonlyrefs,showmanualtags}  

\newcommand{\norm}[1]{\left\Vert#1\right\Vert}

\begin{document}

 \baselineskip 16.6pt
\hfuzz=6pt

\widowpenalty=10000

\newtheorem{cl}{Claim}
\newtheorem{theorem}{Theorem}[section]
\newtheorem{proposition}[equation]{Proposition}
\newtheorem{coro}[equation]{Corollary}
\newtheorem{lemma}[equation]{Lemma}
\newtheorem{definition}[equation]{Definition}
\newtheorem{assum}{Assumption}[section]
\newtheorem{example}[equation]{Example}
\newtheorem{remark}[equation]{Remark}
\renewcommand{\theequation}
{\thesection.\arabic{equation}}

\def\SL{\sqrt H}

\newcommand{\cent}{\operatorname{cent}}

\newcommand{\wid}{\operatorname{width}}
\newcommand{\heit}{\operatorname{height}}

\newcommand{\cdim}{n}

\newcommand{\set}[1]{\mathfrak{#1}}
\newcommand{\vrect}{\set{V}}
\newcommand{\Stack}{\set{S}}
\newcommand{\tile}{\set{T}}

\newcommand{\mar}[1]{{\marginpar{\sffamily{\scriptsize
        #1}}}}

\newcommand{\as}[1]{{\mar{AS:#1}}}
\newcommand\C{\mathbb{C}}
\newcommand\Z{\mathbb{Z}}
\newcommand\R{\mathbb{R}}
\newcommand\RR{\mathbb{R}}
\newcommand\CC{\mathbb{C}}
\newcommand\NN{\mathbb{N}}
\newcommand\ZZ{\mathbb{Z}}
\newcommand\HH{\mathbb{H}}
\def\RN {\mathbb{R}^n}
\renewcommand\Re{\operatorname{Re}}
\renewcommand\Im{\operatorname{Im}}

\newcommand{\mc}{\mathcal}
\newcommand\D{\mathcal{D}}
\def\hs{\hspace{0.33cm}}
\newcommand{\la}{\alpha}
\def \l {\alpha}
\newcommand{\eps}{\tau}
\newcommand{\pl}{\partial}
\newcommand{\supp}{{\rm supp}{\hspace{.05cm}}}
\newcommand{\x}{\times}
\newcommand{\lag}{\langle}
\newcommand{\rag}{\rangle}

\newcommand{\lset}{\left\lbrace}
\newcommand{\rset}{\right\rbrace}

\newcommand\wrt{\,{\rm d}}

\title[]{Schatten classes and commutators of Riesz transform\\ on Heisenberg group and applications}

\author{Zhijie Fan}
\address{Zhijie Fan, School of Mathematics and Statistics, Wuhan University, Wuhan 430072, China}
\email{00033732@whu.edu.cn}

\author{Michael Lacey}
\address{Michael Lacey, Department of Mathematics, Georgia Institute of Technology
Atlanta, GA 30332, USA}
\email{lacey@math.gatech.edu}

\author{Ji Li}
\address{Ji Li, Department of Mathematics, Macquarie University, Sydney}
\email{ji.li@mq.edu.au}

  \date{\today}


\begin{abstract} We study commutators with the  Riesz transforms on the  Heisenberg group $\HH^{n}$.  
The Schatten norm of these commutators is characterized in terms of Besov norms of the symbol.  
This generalizes the  classical Euclidean  results of Peller, Janson--Wolff and Rochberg--Semmes. 
The method of proof extends the earlier methods,  allowing us to address not just the Riesz transforms, but  also the Cauchy--Szeg\H{o} projection and  second order Riesz transforms on $\HH^{n}$ among other settings.
\end{abstract}

\maketitle

\tableofcontents

\section{Introduction}
The commutators with the Riesz transforms are bounded and compact on $L^p(\mathbb R^n)$, $1<p<\infty$, if and only if the symbol $b$ is in the BMO space and VMO space. This is well known, see  \cite{CRW,U}. A finer property of the commutators 
quantifies the Schatten norms, that is the $ \ell ^{p}$ norm of the singular values.  
This was  studied by Peller, for the Hilbert transform on  $\mathbb R$ \cite{P} (see also \cite{P2}). 
And in higher dimensions 
by Janson--Wolff in $\mathbb R^n$, $n\geq2$ \cite{JW}, and later on by Rochberg--Semmes \cite{RS0,RS}. 
The Schatten norm is characterized by the symbol being in certain Besov spaces.  
To summarize the known results are as follows. Let $H$ denote the Hilbert transform and let $R_\ell$ denote the $\ell$-th Riesz transform on $\mathbb R^n$.
\begin{itemize}
\item If $n=1$ and $0< p<\infty$, then $[b,H]$ is in Schatten class $S^p$ if and only if the symbol $b$ is in the Besov space $B_{p,p}^{1/p}(\mathbb R)$.   \cite{P,P2}.
\item Suppose $n\geq2$ and $b\in L^1_{{\rm loc}}(\mathbb R^n)$. When $p>n$, $[b,R_{\ell}]\in S^p$ if and only if $b\in B_{p,p}^{{n}/{p}}(\mathbb R^n)$; when $0<p\leq n$, $[b,R_{\ell}]\in S^p$ if and only if $b$ is a constant \cite{JW,RS}.
\end{itemize}

Notice that the cases of dimensions $ n =1$ and $ n >2$ differ somewhat. This is due to the distinguished nature of the Hilbert transform, 
particularly its close connection to analyticity.   Similar results have been  demonstrated in \cite{FR} for Szeg\H{o} projection, big and little Hankel operators on the unit ball and  Heisenberg group, in \cite{AFP} for the big Hankel operator on Bergman space of the disk, and in \cite{Z} for Hankel operators on the Bergman space of the unit ball. 

The Janson--Wolff inequality has bearing on the a quantised derivative of Alain Connes  introduced in \cite[IV]{Con}. 
In this setting, the (weak) Schatten norm of the commutator is relevant \cite{LMSZ}.    See also some recent progresses in different settings \cite{AP, FX, Is,  Pau, PoSu, PS}.

In this paper we extend the Janson--Wolff result to   commutators of Riesz transforms on Heisenberg groups. 
This requires us to revisit the methods of Janson--Wolff and Rochberg--Semmes, 
replacing  Fourier analytic methods they used with more robust real variable arguments.  
Our result not only recovers the result of Janson--Wolff \cite{JW} and Rochberg--Semmes \cite{RS} on $\mathbb R^n$, $n\geq2$, with the quantitative estimate of the Schatten norm (which was not showed explicitly before), but also opens the door to the study of commutator with certain Calder\'on--Zygmund operators in other important settings beyond $\mathbb R^n$. Examples of such Calder\'on--Zygmund operators include

\begin{enumerate}
\item the Cauchy--Szeg\H{o} projection from Siegel upper half space to its boundary (identified with Heisenberg group), see \cite[Chapter 12, Section 2.4]{St} and \cite{FR}; 

\item  certain second order Riesz transforms, such as the well-known Beurling--Ahlfors operator on the complex plane $\mathbb C$ and  second order Riesz transforms on $\HH^n$. Details will be provided in the last section; 

\item Riesz transforms in the Bessel setting \cite{BCC,Hu} and Neumann Laplacian setting \cite{LW}, which will be addressed in subsequent papers. 
\end{enumerate}

To be more explicit on our result,
let $\HH^{n}$ be Heisenberg group.  It  is a nilpotent Lie group with underlying manifold $\mathbb{C}^{n}\times \mathbb{R}=\{[z,t]:z\in\mathbb{C}^{n}\times\mathbb{R}\}$, the multiplication law
\begin{align}\label{Hn multi law}
[z,t][z^{\prime},t^{\prime}]=[z_{1},\cdots,z_{n},t][z_{1}^{\prime},\cdots,z_{n}^{\prime},t^{\prime}]:=\Big[z_{1}+z_{1}^{\prime},\cdots,z_{n}+z_{n}^{\prime},t+t^{\prime}+2{\rm Im}\Big(\sum_{j=1}^{n}z_{j}\overline{z}_{j}\Big)\Big]
\end{align}
and
the homogeneous norm $\rho(g)$ (details on the notation will be given in Section \ref{Preliminaries}).

For any $\ell=1,2,\ldots,2n$, let $R_{\ell}$ be the Riesz transform on Heisenberg groups $\HH^{n}$
and the commutator with $R_{\ell}$ is defined as follows.
$$[b,R_{\ell}](f)(x):= b(x)R_{\ell}(f)(x) - R_{\ell}(bf)(x).  $$
Next we recall the definition of the homogeneous Besov space in the following form.
\begin{definition}
{Suppose $1<p,q< \infty$} and $0<\alpha<1$. Let $f\in L_{\rm loc}^{1}(\HH^{n})$. Then we say that $f$ belongs to Besov space $B_{p,q}^{\alpha}(\HH^{n})$ if
\begin{align*}
{\int_{\HH^{n}}\frac{\|f(g \cdot )-f(\cdot)\|_{L^{p}(\HH^{n})}^{q}}{\rho(g)^{2n+2+q\alpha}}dg<\infty.}
\end{align*}
\end{definition}

 We recall the definition of the Schatten class $S^{p}$. Note that if $T$ is any compact operator on $L^{2}(\HH^{n})$, then $T^{*}T$ 
 is compact, symmetric and  positive. It is  diagonalizable. For $0<p<\infty$, we say that $T\in S^{p}$ if $\{\lambda_{n}\}\in \ell^{p}$, where $\{\lambda_{n}\}$ is the sequence of square roots of eigenvalues of $T^{*}T$ (counted according to multiplicity).

Our main theorem is the following.
\begin{theorem}\label{schatten}
Suppose that $0<p<\infty$ and $b\in  L^1_{{\rm loc}}(\HH^n)$. Then for any $\ell\in \{1,2,\cdots,2n\}$, one has  $[b,R_{\ell}]\in S^p$ 
if and only if 

\begin{enumerate}
\item $b\in B_{p,p}^{\frac{2n+2}{p}}(\HH^n)$, if $p>2n+2$; in this case we have $\|b\|_{B_{p,p}^{\frac{2n+2}{p}}(\HH^n)}\approx \|[b,R_{\ell}]\|_{S^p}$; 

\item $b$ is a constant, if\ $0<p\leq 2n+2$.

\end{enumerate}

\end{theorem}

In the Euclidean setting, the Riesz transforms have an explicit form,   $\Omega(x)\over |x|^n$ where $n$ is the dimension of the underlying space and $\Omega(x)$ is a smooth homogeneous function of degree 0.  
This leads to arguments highly dependent on the form of the kernel.  
However, the Riesz transform kernel on Heisenberg group has no such convenient form. 
And, our argument depends upon recent developments.  
A pointwise lower bound of the Riesz transform kernel on stratified Lie groups (which covers the Heisenberg group) was established in \cite{DLLW}  to characterize the boundedness of the commutator of Riesz transform.  
We have to further develop this theme to prove the main result. See Theorem~\ref{nondegen} below.  
Indeed, Theorem~\ref{nondegen} is key to our proof, a canonical `non-degenerate' condition.    
It depends upon the kernel of the Riesz transforms only being zero on a set of zero measure, 
and being suitably large.  
In addition, the property aligns well with the  martingale structure on Heisenberg groups.  
Verifying this property should be central in   settings beyond the Euclidean.   We return to this point in \S\ref{appl}.


Our proof uses a natural martingale structure on the Heisenberg group, and an associated Haar basis, 
and crucially a notion of \emph{nearly weakly orthogonal} due to Rochberg--Semmes \cite{RS}.  
It is very well adapted to the analysis of Schatten norms in Harmonic Analysis settings.  See \eqref{e:NWO}.  
As with other methods, the median of the symbol on the atoms of the martingale is important.   


The paper is organized as follows. In Section \ref{Preliminaries}, we recall the tilings and Haar Basis on Heisenberg group and characterization of Schatten class. In Section \ref{Rieszsection} we recall the basic property for Riesz transform and then prove  the pointwise lower bound for the Riesz kernel (Theorem \ref{nondegen}). In Sections \ref{three} and \ref{four}, we give   the proof of Theorem \ref{schatten} for the cases $p>2n+2$ and $0<p\leq 2n+2$, respectively, which lies in Propositions \ref{schattenlarge1}, \ref{schattenlarge2} and \ref{mainprop}. In Section \ref{appl}, we provide the applications of our approach to some well-known Calder\'on--Zygmund operators beyond the Euclidean setting.

{Throughout the paper we denote the $L^{p}(\HH^n)$ norm of a function $f$ by $\|f\|_{p}$, $1\leq p\leq \infty$. Other norms (such as the Besov norm or Schatten norm) are given explicitly in the context. The
indicator function of a subset $E\subseteq X$ is denoted by $\chi_{E}$. We use $A\lesssim B$ to denote the statement that $A\leq CB$ for some constant $C>0$, and $A\approx B$ to denote the statement that $A\lesssim B$ and $B\lesssim A$.}

\section{Preliminaries on $\HH^n$}\label{Preliminaries}
\setcounter{equation}{0}

Let $\HH^{n}$ be a Heisenberg group, which is a nilpotent Lie group with underlying manifold $\mathbb{C}^{n}\times \mathbb{R}=\{[z,t]:z\in\mathbb{C}^{n}\times\mathbb{R}\}$ and multiplication law as in \eqref{Hn multi law}.
Then the identity of $\HH^{n}$ is the origin and the inverse is given by $[z,t]^{-1}=[-z,-t]$. In addition to the Heisenberg group multiplication law, for each positive number $\lambda$, non-isotropic dilations $\delta_{\lambda}$ on $\HH^{n}$ are given by
$$\delta_{\lambda}(g):=\delta_{\lambda}[z,t]:=[\lambda z,\lambda^{2}t].$$
Besides, the norm structure $\rho$ on $\HH$ is defined by
$$\rho(g)=\rho([z,t])=\max\{|z|_{\mathbb C^n},|t|^{1/2}\},$$
where $|z|_{\mathbb C^n}^2=\sum_{j=1}^{n} |z_j|^2$. The Haar measure on $\HH^{n}$ coincides with Lebesgue measure on $\mathbb{R}^{2n+1}$. For any measurable set $E\subset \HH^{n}$, $|E|$ denotes its Haar measure. It is direct to see that $\rho(g^{-1})=\rho(-g)=\rho(g)$ and $ \rho(\delta_{\lambda}(g))=\lambda \rho(g)$.

The $2n+1$ vector fields
$$X_{\ell}:=\frac{\partial}{\partial x_{\ell}}-2y_{\ell}\frac{\partial}{\partial t},\ \ Y_{\ell}:=\frac{\partial}{\partial y_{\ell}}+2x_{\ell}\frac{\partial}{\partial t},\ \  \mathcal T:=\frac{\partial}{\partial t},\ \ \ell=1,2,\cdots,n$$
form a natural basis for the Lie algebra of left-invariant vector field on $\HH^{n}$. For convenience, we set $X_{n+\ell}:=Y_\ell, \ell=1,2,\cdots,n$ and set $X_{2n+1}:=\mathcal T$. The standard sub-Laplacian $\Delta_{\HH}$ on the Heisenberg group is defined by $\Delta_{\HH}:=\sum_{\ell=1}^{2n}X_{\ell}^{2}$. For any multi-index $I=(i_{1},\cdots,i_{2n+1})\in\mathbb{N}^{2n+1}$, we set $X^{I}:=X_{1}^{i_{1}}X_{2}^{i_{2}}\cdots X_{2n+1}^{i_{2n+1}}$ and further set
$$|I|:=i_{1}+\cdots+i_{2n+1} \quad   \textup{ and}  \quad \ d(I):=i_{1}+\cdots+i_{2n}+2i_{2n+1}.$$
The integers  $I$ and $d(I)$ are said to be the topological degree and homogeneous degree of the differential $X^{I}$, respectively.

\subsection{Tiles on $\HH^n$}\label{tilesection}

{We recall the metrics and tilings in $\HH^n$ summarized in} \cite{CCLLO}.
We shall use the \emph{gauge distance} $d$, which is defined by setting
\begin{equation}\label{eq:gauge-distance}
\begin{aligned}
d(g, g') := \norm{ g'^{-1} \cdot g } = \norm{ g^{-1} \cdot g'},
\qquad\forall g, g' \in \HH^n ,
\end{aligned}
\end{equation}
where $\norm{{}\cdot{}}$ is given by
\begin{equation}\label{eq:gauge-norm}
\begin{aligned}
\norm{ (z,t) }
:= \max\lset |x_1|, |y_1|, \dots, |x_\cdim |, |y_\cdim |, |t|^{1/2} \rset
\qquad\forall (z,t)\in\HH^n .
\end{aligned}
\end{equation}
It is easy to see that $d$ is equivalent to the homogeneous norm $\rho$.  See \cite[Section 2.2]{Ty} for a discussion.
We write $B(g, r)$ for the ball in $\HH^n$ with center $g$ and radius $r$ constructed using the distance $d$.
We also use balls in the (algebraic) center of $\HH^n$, which may be identified with $\R$: we define $B^{*}(t, s) := \{ t' \in \R : |t-t'| < s \}$.
Tubes are sets of the form $g \cdot B(o, r) \cdot B^{*}(0, s)$, which are images of products of balls in $\HH^n  \times \R$ under the multiplication in \eqref{Hn multi law}.
We  recall that $T(g,r,s)$ is defined as
\begin{equation}\label{eq:def-tube}
T(g,r,s)
:= g \cdot B(o, r) \cdot B^{*}(0, s)
= B(g, r) \cdot B^{*}(0, s).
\end{equation}

We use the work of \cite{Str, Ty} on self-similar tilings to find a ``nice'' decomposition of $\HH^n$, analogous to the decomposition of $\R^n$ into dyadic cubes in classical harmonic analysis, and describe an analogue of a lemma of Journ\'e \cite{J}.
We identify $\C^\cdim$ with $\R^{2\cdim}$, $|z|_\infty$ denotes $\max\{ |x_1|, |y_1|, \dots |x_\cdim |, |y_\cdim | \}$, $Q_0$ denotes the cube $[-1/2,1/2)^{2\cdim}$, and $\HH^n_{\Z}$ denotes the subgroup $\{ (z,t) \in \HH^n : z \in \Z^{2\cdim}, t \in (2\cdim)^{-1} \Z\}$.


{\begin{theorem}[\cite{Str, Ty}]
There is a measurable function $f: Q_0 \to \R$ such that $f(0) = {1\over2(n+1)}$ and
\[
\frac{1}{4\cdim (\cdim + 1)}
\leq f(z)
\leq \frac{2\cdim + 1}{4\cdim (\cdim + 1)}
\qquad\forall z \in Q_0,
\]
such that the set $T_o$, defined by
\[
T_o := \lset (z,t) : z \in Q_0, f(z) - \frac{1}{2\cdim} \leq t < f(z) \rset,
\]
has the property that
\[
\delta_{2\cdim+1} (T_o) = \bigcup_{g \in \Delta} g \cdot T_o ,
\]
where $\Delta := \{ (z,t) \in \HH^n_{\Z} : |z|_\infty \leq \cdim : |t| \leq \cdim + 1 \}$.
\end{theorem}}

The definitions of $T_o$ and the metrics that we use show that
\begin{equation}\label{eq:size-of-basic-tile}
\begin{aligned}
T_o
\subset \lset(z,t) \in \HH^n: |z|_\infty \leq 1/2, |t| \leq 3/8 \rset
\subseteq \bar B(o, 1/2)\cdot \bar B^{*}(o, 1/8)
= \bar T(o,1/2,1/8),
\end{aligned}
\end{equation}
where the barred symbols indicate closures.
We note that $|T_o| = 1/2\cdim$ while $|T(o,1/2,1/8)| = 3/4$.

\begin{definition}\label{def:tiles}
We define
\[
\tile_0 := \{ g \cdot T_o : g \in \HH^n_{\Z} \},
\qquad
\tile_j := \delta_{(2\cdim+1)^j} \tile_0
\quad\text{and}\quad
\tile := \bigcup_{j \in \Z} \tile_j .
\]
We call the sets $T \in \tile$ \emph{tiles}.
If $j \in \Z$ and $g \in \HH^n_{\Z}$ and $T = \delta_{(2\cdim+1)^j} (g \cdot T_o)$, then $T = \delta_{(2\cdim+1)^j} (g) \cdot \delta_{(2\cdim+1)^j} (T_o)$, and we further define
\[
\cent(T) := \delta_{(2\cdim+1)^j} (g),
\qquad
\wid(T) := (2\cdim+1)^j
\quad\text{and}\quad
\heit(T) := \frac{(2\cdim+1)^{2j}}{2\cdim} \,.
\]
And we define $I_{j}$ be the $j$-th center set consisting of all the centers of $T\in \tile_j$. That is,
$$I_{j}=\{\cent(T):T\in\tile_{j}\}.$$
\end{definition}

\begin{lemma}[\cite{Ty,Str}]\label{thm:Heisenberg-grid}
Let $\tile_j$ and $\tile$ be defined as above.
Then the following hold:
\begin{enumerate}
  \item for each $j \in \Z$, $\tile_j$ is a partition of $\HH^n$, that is, $\HH^n = \bigcup_{T \in \tile_j} T$;
  \item $\tile$ is nested, that is, if $T, T' \in \tile$, then either $T$ and $T'$ are disjoint or one is a subset of the other;
  \item for each $j \in \Z$ and $T\in\tile_j$, $T$ is a union of $(2n+1)^{2n+2}$ disjoint congruent subtiles in $\tile_{j-1}$;
  \item $B(g, C_1 q) \subseteq T \subseteq B(g, C_2 q)$, where $g = \cent(T)$ and $q = \wid(T)$ for each $T \in \tile$; the constants $C_1$ and $C_2$ depend only on $\cdim$;
  \item if $T \in \tile_j$, then $g \cdot T \in \tile_j$ for all $g \in \delta_{(2\cdim+1)^j} \HH^n_{\Z}$, and $\delta_{(2\cdim+1)^k} T \in \tile_{j+k}$ for all $k \in \Z$.
\end{enumerate}
\end{lemma}

Every tile is a dilate and translate of the basic tile $T_o$, so all have similar geometry.
Hence each tile in $\tile_j$ is a fractal set---its boundary is a set of Lebesgue measure $0$ and (Euclidean Hausdorff) dimension $2\cdim$---and is ``approximately'' a Heisenberg ball of radius $(2\cdim+1)^{j}$.
The decompositions are \emph{product-like} in the sense that the tiles project
onto cubes in the factor $\C^\cdim$, and their centers form a product set. If two tiles in $\tile_j$ are ``horizontal neighbors'', then the distance between their centers is $(2\cdim+1)^{j}$, while if they are ``vertical neighbors'', then the distance is $(2\cdim+1)^{2j}/2\cdim$.

\subsection{An Explicit Haar Basis on Heisenberg group}

Next we recall the explicit construction in \cite{KLPW} of a Haar basis. Note that in \cite{KLPW}, the Haar basis was constructed on a system of dyadic cubes for general metric space with a positive Borel measure. Here we apply it to the specific setting of Heisenberg group $\HH^n$ on the system of tiles.

There exists a Haar basis on $\HH^n$:
$\{h_{T}^{\epsilon}: T\in \tile, \epsilon = 1,\dots,M_n - 1\}$ for $L^p(\HH^n)$,
$1 < p < \infty$, where $M_n:=\# \mathfrak{H}(T)= (2n+1)^{2n+2}$ denotes the number of
sub-tiles of $T$ and $\mathfrak{H}(T)$ denotes the collection of sub-tiles of $T$.

\begin{lemma}[\cite{KLPW}]\label{thm:convergence}
For each $f\in
    L^p$, we have
    \[
        f(x)
        =  \sum_{T\in\tile}\sum_{\epsilon=1}^{M_n-1}
            \langle f,h^\epsilon_T\rangle h^\epsilon_T(x), 
    \]
    where the sum converges (unconditionally) both in the
    $L^p$-norm and pointwise almost everywhere.
 \end{lemma}

The following theorem collects several basic properties of the
functions $h_{T}^{\epsilon}$.

\begin{lemma}[\cite{KLPW}]\label{prop:HaarFuncProp}
    The Haar functions $h_{T}^{\epsilon}$, $T\in\tile$,
    $\epsilon = 1,\ldots,M_n - 1$, have the following properties:
    \begin{itemize}
        \item[(i)] $h_{T}^{\epsilon}$ is a simple Borel-measurable
            real function on $\HH^{n}$;
        \item[(ii)] $h_{T}^{\epsilon}$ is supported on $T$;
        \item[(iii)] $h_{T}^{\epsilon}$ is constant on each
            $R\in\mathcal{H}(T)$;
        \item[(iv)] $\int_T h_{T}^{\epsilon}(g)\, dg = 0$ (cancellation);
        \item[(v)] $\langle h_{T}^{\epsilon},h_T^{\epsilon'}\rangle = 0$ for
            $\epsilon \neq \epsilon'$, $\epsilon$, $\epsilon'\in\{1, \ldots, M_n - 1\}$;
        \item[(vi)] the collection
            $
                \big\{|T|^{-1/2}\chi_T\big\}
                \cup \{h_{T}^{\epsilon} : \epsilon = 1, \ldots, M_n - 1\}
            $
            is an orthogonal basis for the vector
            space~$V(T)$ of all functions on $T$ that is a constant on each sub-cube $R\in\mathfrak{H}(T)$;
        \item[(vii)] 
        if $h_{T}^{\epsilon}\not\equiv 0$ then
            $
                \|h_{T}^{\epsilon}\|_{p}
                \approx |T|^{\frac{1}{p} - \frac{1}{2}}
                \quad \text{for}~1 \leq p \leq \infty;
            $
        \item[(viii)] 

                $\|h_{T}^{\epsilon}\|_{1}\cdot
                \|h_{T}^{\epsilon}\|_{\infty} \approx 1$.
    \end{itemize}
\end{lemma}

\subsection{Characterization of Schatten class}

The Schatten norm is defined in a non-linear fashion. Estimating it above, and below, is not necessarily straight forward.  
Operators with kernels, such as commutators, admit general upper bounds in terms of norms on the kernels. 
These general facts are recalled, and used, in \S\ref{s:Sufficient}.   

Characterizations of Schatten norms for general operators are well known, and frequently expressed in terms of 
supremums, or infimums, over all choices of orthonormal bases for the Hilbert space in question.  

Rochberg and Semmes \cite{RS} proposed a notion of \emph{nearly weakly orthogonal  (NWO)} sequences of functions. 
This notion is closely connected to Carleson measures.  For our purposes, we do not need to recall the full definition 
of NWO sequences.  With the development of tiles in \S\ref{tilesection},  we have the inequality below, for any 
bounded compact operator $ A $ on $ L ^2 (\mathbb H ^{n})$: 
\begin{equation}\label{e:NWO}
\Bigl[
\sum_{T\in \tile } \lvert  \langle A e_T, f_T \rangle\rvert ^{p} 
\Bigr] ^{1/p} \lesssim \lVert A \rVert _{S ^{p}},
\end{equation}
where $\{e_T\}_T$ and $\{f_T\}_T$ are function sequences satisfying $ \lvert  e_T\rvert,  \lvert  f_T\rvert  \leq \lvert  T\rvert ^{-1/2} \chi_{T} $.  
This inequality can be found in  \cite[(1.10), \S3]{RS}.  

%
%
%
%
%

\section{Lower bound of the Riesz transform kernel on $\HH^n$}\label{Rieszsection}
\setcounter{equation}{0}

For any $\ell=1,2,\ldots,2n$, The the Riesz transform on Heisenberg groups $\HH^{n}$ is  given by $R_{\ell}=X_{\ell} (-\Delta_{\HH})^{-1/2}$.
It is well known that the heat kernel $p_h$ on $\HH^n$ has this 
form (cf. \cite{G77}): for $g=[z,t]\in \HH^n$,
\begin{eqnarray} \label{hkf}
p_h(g) = \frac{1}{2 (4 \pi h)^{n+1}} \int_{\R}
\exp{\Big(\frac{\lambda}{4 h} ( t\, \i  - | z |_{\mathbb C^n}^2 \coth{\lambda})
\Big)} \Big( \frac{\lambda}{\sinh{\lambda}} \Big)^n \, d\lambda,\quad \i^2=-1.
\end{eqnarray}
Moreover,  $p_h$ on $\HH^n$ satisfies (c.f. for example \cite[Equation (1.73)]{FoSt})
\begin{align} \label{hkp1}
p_h(g) = h^{-n-1} p(\delta_{\frac{1}{\sqrt{h}}}(g)), \qquad  \forall h > 0, \  g \in \HH^n.
\end{align}

The kernel of the $\ell^{{\rm th}}$  Riesz transform $R_\ell $ ($1 \leq \ell \leq  2n$) is written simply as $K_\ell(g)$. It is well-known that
$ K_\ell \in C^{\infty}(\HH^n \setminus \{o\})$, and it satisfies the scaling condition 
\begin{align} \label{kjs}
 \ K_\ell(\delta_r(g)) = r^{-2n-2} K_\ell(g), \quad \forall g \neq o, \ r > 0, \ 1 \leq \ell \leq  2n. 
\end{align}
Indeed, this follows from the relationship between the Riesz transform and heat kernel  \eqref{hkp1} given by 
\begin{align*}
K_\ell(g) = \frac{1}{\sqrt{\pi}} \int_0^{+\infty} h^{-\frac{1}{2}} X_\ell p_h(g) \, dh  = \frac{1}{\sqrt{\pi}} \int_0^{+\infty} h^{- n - 2} \left( X_\ell p \right)(\delta_{\frac{1}{\sqrt{h}}}(g)) \, dh.
\end{align*}
We  recall that
by the classical estimates for heat kernel and its derivations on stratified groups (see for example \cite{VSC}), 
it is well-known that (e.g. \cite{FoSt})\ for any multi-index $I=(i_1,\cdots,i_{2n})\in \mathbb{N}^{2n}$, $\forall 1 \leq \ell \leq 2n$, the Riesz transform kernel satisfies the following smoothness inequality:
\begin{align} \label{MEHT}
 |X^I K_\ell(g)| \lesssim \rho(g)^{-2n-2-|I|}.
\end{align}
We now establish the following fundamental result for the pointwise lower bound of the Riesz transform kernel, which is one of the key property for proving our main theorem. 
It is of independent interest, in that this property can be seen to hold for other  Calder\'on--Zygmund operators.

\begin{theorem}\label{nondegen}
There exists a positive integer $A_0$ such that:\ \  for each fixed $N\in\mathbb{N}\cup\{0\}$,

$\bullet$ for any $T\in \tile_{j}$, there is a unique $T_{N+A_0}\in \tile_{N+j+A_0}$ such that $T\subset T_{N+A_0}$.

$\bullet$ furthermore, for each $\ell\in \{1,2,\cdots,2n\}$, there exist positive constants $3\leq A_{1}\leq A_{2}$ and $C>0$ such that for any tile $T\in\tile_{j}$  and $N\in\mathbb{N}$, there exists a tile $\hat{T}\in\tile_{j}$ satisfying:
\begin{enumerate}
  \item $\hat{T}\subset T_{N+A_0}$;

  \item $A_{1}(2n+1)^{N+j}\leq d(\cent{(T)},\cent(\hat{T}))\leq A_{2}(2n+1)^{N+j}$;

  \item for all $(g,\hat g)\in T\times \hat{T}$, $K_{\ell}((\hat g)^{-1} g)$ does not change sign;

  \item for all $(g,\hat g)\in T\times \hat{T}$, $|K_{\ell}((\hat g)^{-1} g)|\geq C (2n+1)^{-(2n+2)(N+j)}$.
\end{enumerate}
\end{theorem}
\begin{proof}
Begin with  this fundamental fact of the Riesz transform kernel from \cite[Theorem 1.5]{DLLWW}:  
$$ K_\ell(g)\not=0\qquad {\rm a.e.}\ g\in \mathbb H^n , {\rm\ \ for\ each\ fixed\ } \ell\in\{1,2,\ldots,2n\}. $$
From the scaling property of $K_\ell$ (c.f. \eqref{kjs}) and the property above, we obtain that
$$ \quad K_\ell( g)\not=0\qquad {\rm a.e.}\ \  g\in \mathbb S^n,  $$
where $\mathbb S^n=\{g\in \HH^n:\ \rho( g)=1\}$ is the unit sphere in $\HH^n$.
Let $E_\ell:=\{ g\in \mathbb S^n:\   K_\ell( g)=0 \}$. Then $\sigma(E_\ell)=0$, where $\sigma$ represents the surface measure,  and for every small positive number
$\epsilon$, there exists an open set $\mathcal E_\ell$ covering $E_\ell$ such that $\sigma(\mathcal E_\ell)<\epsilon$.
Since $K_\ell$ is a $C^\infty$ function in $\HH^n \backslash \{o\}$, there exists $g_\ell$ in $\HH^n$ with $\rho(g_
\ell)=1$ such that
$$ |K_\ell(g_\ell)|=\min_{g\in \mathcal F_\ell } |K_\ell(g) |>0, $$
where $\mathcal F_\ell:=\mathbb S^n\backslash \mathcal E_\ell$.

Hence, there exists $ 0<\varepsilon_o\ll1$ such that
\begin{align}\label{non zero e1}
 |K_\ell( g)|>{1\over 2} |K_\ell( g_\ell)|
\end{align}
for all $g\in B(\mathcal F_\ell, 4\varepsilon_o) = \{g\in\HH^n: \exists \tilde g\in \mathcal F_\ell {\rm\ \ such \ that} \ \ d(g,\tilde g)<4\varepsilon_o   \}$.

\smallskip 
We now turn to the tiles.  Based on the construction of tiles,  for every $T\in \tile_{j}$ and for each fixed $N\in\mathbb{N}$, there exists a unique $T_{N+A_0}\in 
\tile_{N+j+A_0}$ such that $T\subset T_{N+A_0}$. Here $A_0$ is a  positive integer to be determined later.
We now fix $N\in\mathbb N$ and choose an arbitrary $T\in \tile_{j}$.


We first {\bf claim} that for the chosen $T\in \tile_{j}$ and the unique tile $T_{N+A_0}\in \tile_{N+j+A_0}$ with $T\subset T_{N+A_0}$, there must be some $\hat g \in T_{N+A_0}$ with
$d(h, \hat g) = \mathfrak C (2n+1)^{N+j+A_0}$ and
$d(\hat g, T_{N+A_0}^c) >10C_2(2n+1)^j$
such that
\begin{align}\label{ee claim}
(\delta_{\mathfrak C^{-1}(2n+1)^{-N-j-A_0}} (h^{-1} \hat g) )^{-1}\in \mathcal F_\ell,
\end{align}
where $h= \cent{(T)}$, $\mathfrak C$ is a positive constant such that ${ C_1\over2}<\mathfrak C < {3 C_1\over4}$,  $C_1$ and $C_2$ are the constants in Lemma \ref{thm:Heisenberg-grid}.


We now prove this claim. Suppose that for all $\hat g \in T_{N+A_0}$ with
$d(h, \hat g) = \mathfrak C (2n+1)^{N+j+A_0}$ and
$d(\hat g, T_{N+A_0}^c) >10C_2(2n+1)^j$, \eqref{ee claim} does not hold. Then since $\rho((\delta_{\mathfrak C^{-1}(2n+1)^{-N-j-A_0}} (h^{-1} \hat g))^{-1})=1$,
 we obtain
that $(\delta_{\mathfrak C^{-1}(2n+1)^{-N-j-A_0}} (h^{-1} \hat g))^{-1} \in \mathcal S_\ell.$ However, due to the construction of the system of tiles, we obtain that
$$ {\ \sigma( \{\hat g \in T_{N+A_0}:\
d(h, \hat g) = \mathfrak C (2n+1)^{N+j+A_0},\ d(\hat g, T_{N+A_0}^c) >10C_2(2n+1)^j\} )\ \over \sigma(\{\hat g \in \HH^n:\
d(h, \hat g) = \mathfrak C (2n+1)^{N+j+A_0}\})} > \mathfrak D>0, $$
where $ \mathfrak D\in(0,1)$ is a constant depending on $n$, $N$ and $A_0$ only, but independent of $ j$ and $T$. This contradicts to the fact that $\sigma(\mathcal S_\ell)<\epsilon$ for any  small positive $\epsilon$ given at the beginning. Thus, the claim holds.

Now based on the claim, we choose $\hat h \in T_{N+A_0}$ with
$d(h, \hat h) = \mathfrak C (2n+1)^{N+j+A_0}$ and
$d(\hat h, T_{N+A_0}^c) >10C_2(2n+1)^j$ such that $(\delta_{\mathfrak C^{-1}(2n+1)^{-N-j-A_0}} (h^{-1} \hat h))^{-1} \in \mathcal F_\ell$. Let $\tilde g_\ell :=(\delta_{\mathfrak C^{-1}(2n+1)^{-N-j-A_0}} (h^{-1} \hat h))^{-1} $.
Without lost of generality, we assume that $K_\ell(\tilde g_\ell)$ is positive.

From the definition of $\tilde g_\ell $ we see that
\begin{align}\label{g*}
 \hat h= h \cdot \delta_{\mathfrak C (2n+1)^{N+j+A_0}}( \tilde g_\ell^{-1} ).
\end{align}
%
%
%
%
%
Next, we choose the integer $A_0$ so that $(2n+1)^{N+A_0}> 5C_2\mathfrak C^{-1} \varepsilon_0^{-1}$. Then fix some $\eta\in(0, 2 \varepsilon_o)$ such that
 the two balls $B(h,  \eta r)$ and $B( \hat h,  \eta r) $ with $r=\mathfrak C (2n+1)^{N+j+A_0}$ satisfy the following condition:
$$5 C_2(2n+1)^j< \eta r< 10 C_2(2n+1)^j.$$
Then we can deduce that $T\subset B(h,  \eta r)$ and $B( \hat h,  \eta r) \subset T_{N+A_0}$.

 It is direct that
 for every $g\in B(h,  \eta r)$, we can write
 $$ g =h\cdot \delta_r (g'_1) $$
 where $g'_1 \in B(o,  \eta)$.
Similarly, for every $\hat{g}\in B( \hat h,  \eta r)$, we can write
 $$ \hat{g} = \hat h\cdot \delta_r (g'_2) $$
 where $g'_2 \in B(o,  \eta)$.

As a consequence, we have
\begin{align}\label{dilation}
K_\ell(g,\hat{g}) &= K_\ell\big( h\cdot \delta_r (g'_1) ,   \hat h\cdot \delta_r (g'_2)   \big)\\
&= K_\ell\big(   h\cdot \delta_r (g'_1) ,   h \cdot \delta_r( \tilde g_\ell^{-1} )\cdot \delta_r (g'_2)   \big)\nonumber\\
&= K_\ell\big(    \delta_r (g'_1) ,     \delta_r( \tilde g_\ell^{-1} )\cdot \delta_r (g'_2)   \big)\nonumber\\
&= K_\ell\big(    \delta_r (g'_1) ,     \delta_r( \tilde g_\ell^{-1} \cdot g'_2)   \big)\nonumber\\
&= r^{-2n-2} K_\ell\big(    g'_1,     \tilde g_\ell^{-1} \cdot g'_2  \big)\nonumber\\
&= r^{-2n-2} K_\ell\big(    (g'_2)^{-1} \cdot    \tilde g_\ell \cdot g'_1  \big),\nonumber
\end{align}
where the second equality comes from \eqref{g*}, the third comes from the property of the left-invariance
and the fifth comes from \eqref{kjs}.

Next, we note that
\begin{align*}
d\big(    (g'_2)^{-1} \cdot    \tilde g_\ell \cdot g'_1, \tilde g_\ell \big) &= d\big(      \tilde g_\ell \cdot g'_1, g'_2 \cdot  \tilde g_\ell \big)\\
&\leq  \, \left[ d\big(      \tilde g_\ell \cdot g'_1,   \tilde g_\ell \big)+ d\big(      \tilde g_\ell , g'_2 \cdot  \tilde g_\ell \big) \right]\\
&=  \, \left[  d\big(  g'_1,   o \big)+ d\big( o, g'_2 \big) \right]\\
&\leq 2  \eta\\
&<4  \varepsilon_o,
\end{align*}
which shows that $ (g'_2)^{-1} \cdot    \tilde g_\ell \cdot g'_1$ is contained in the ball $B(\tilde g_\ell, 4 \varepsilon_o)$
for all $g'_1 \in B(o,  \eta)$ and for all $g'_2 \in B(o,  \eta)$.

Thus, from  \eqref{non zero e1}, we obtain that
\begin{align}\label{lower bound e1}
| K_\ell\big(    (g'_2)^{-1} \cdot    \tilde g_\ell\cdot g'_1  \big)| > {1\over 2 } | K_\ell(\tilde g_\ell)|
\end{align}
and for all $g'_1 \in B(o,  \eta)$ and for all $g'_2 \in B(o,  \eta)$,
$K_\ell\big(    (g'_2)^{-1} \cdot    \tilde g_\ell \cdot g'_1  \big)$ and $K_\ell(\tilde g_\ell)$ have the same sign.

Now
combining the equality \eqref{dilation} and   \eqref{lower bound e1} above, we obtain that
\begin{align}\label{lower bound e2}
|K_\ell(g,\hat{g})|   > {1\over 2 } r^{-2n-2} |K_\ell(\tilde g_\ell)|
\end{align}
for every $g\in B(h,  \eta r)$ and for every $\hat g\in B( \hat h,  \eta r)$, where $K_\ell(g,\hat{g})$ and $K_\ell(\tilde g_\ell)$ have the same sign. Here $K_\ell(\tilde g_\ell)$ is a fixed
constant independent of $\eta$, $r$, $h$, $g_1$ and $g_2$. We denote
$$ C(\ell,n)= {1\over 2}|K_\ell(\tilde g_\ell)|.$$

From the lower bound \eqref{lower bound e2} above, we further obtain that
for the suitable $\eta\in (0,\varepsilon_o)$,
\begin{align}\label{lower bound e3}
|K_\ell(g,\hat{g})|   > C(\ell,n) r^{-2n-2}
\end{align}
for every $g\in B(h,  \eta r)$ and for every $\hat{g}\in B( \hat h,  \eta r)$. Moreover,
the sign of $K_\ell (g,\hat{g})$ is invariant
for every $g\in B(h,  \eta r)$ and for every $\hat{g}\in B( \hat h,  \eta r)$.

Based on the fact that $B( \hat h,  \eta r)\subset T_{N+A_0}$ and $\eta r> 5 C_2(2n+1)^j $, there must be some tile $\hat T \in \tile_{j}$ such that $\hat{T}\subset B( \hat h,  \eta r)$. Also note that $T\subset B( h,  \eta r)$. Hence we obtain that  $A_{1}(2n+1)^{N+j}\leq d(\cent{(T)},\cent(\hat{T}))\leq A_{2}(2n+1)^{N+j}$, where $A_1$ and $A_2$ depends only on $A_0$ and $\mathfrak C$. Moreover, we see that
  for all $(g,\hat g)\in T\times \hat{T}$, $K_{\ell}((\hat g)^{-1} g)$ does not change sign and that
   for all $(g,\hat g)\in T\times \hat{T}$, $|K_{\ell}((\hat g)^{-1} g)|\gtrsim  (2n+1)^{-(2n+2)(N+j)}$, where the implicit constant depends on $C(\ell,n)$ and $A_0$.

The proof of Theorem \ref{nondegen} is complete.
\end{proof}

\section{Theorem \ref{schatten}:   $2n+2<p<\infty$}\label{three}
\setcounter{equation}{0}

\subsection{Proof of the necessary condition}
In this subsection, we assume that $[b,R_{\ell}]\in S^p$ for some $2n+2<p<\infty$ and then prove that $b\in B_{p,p}^{\frac{2n+2}{p}}(\HH^{n})$.

We need these preliminary observations.  
Let $\tile_{k}$ be the decomposition of $\HH^{n}$ into tiles $T$ as in Section \ref{tilesection}. We define the conditional expectation of a locally integrable function $f$ on $\HH^{n}$ with respect to the increasing family of $\sigma-$algebras $\sigma(\tile_{-k})$ by the expression: $$E_{k}(f)(g)=\sum_{T\in\tile_{-k}}(f)_{T}\chi_{T}(g),\ g\in\HH^n.$$
where we denote $(f)_{T}$ be the average of $f$ over $T$, that is, $(f)_{T}:=\fint_{T}f(g)dg:=\frac{1}{|T|}\int_{T}f(g)dg$.

For $T\in\tile_{k}$, we let $h_{T}^{1}$, $h_{T}^{2},\ldots, h_{T}^{M_n-1}$ be a family of Haar functions defined in Lemma \ref{prop:HaarFuncProp}. Next, we choose $h_{T}$ among these functions such that
$$h_{T}=\left\{h_{T}^{\epsilon}:\left|\int_{T}f(g)h_{T}^{\epsilon}(g)\,dg\right|{\rm \ is\ maximal\ with\ respect\ to\ }\epsilon=1,2,\ldots,M_{n}-1 \right\}.$$

Note that the function $(E_{k+1}(f)(g)-E_{k}(f)(g))\chi_T(g)$ is a sum of $M_n$ Haar functions. That is, we are in a finite dimensional setting and all $L^p$-spaces have comparable norms. So
we have that
\begin{align}\label{tttt1}
\left(\fint_{T}|E_{k+1}(f)(g)-E_{k}(f)(g)|^{p}\,dg\right)^{1/p}
&\leq C  |T|^{-1/2}\left|\int_{T}f(g)h_{T}(g)\,dg\right|,
\end{align}
where $C$ is a constant only depending on $p$ and $n$.

This is the main Lemma.  
\begin{lemma}\label{step1}
Let $1< p<\infty$ and suppose that $b$ is a locally integrable function satisfying $\|[b,R_{\ell}]\|_{S^p}<\infty$, then there exists a constant $C>0$ such that for any $k\in\mathbb{Z}$,
\begin{align} \label{e:step1}
\sum_{k}(2n+1)^{(2n+2)k}\| E _{k+1}(b) -E_{k}(b)\|_{p} ^{p} \lesssim \|[b,R_{\ell}]\|_{S^p} ^{p}
\end{align}
\end{lemma}

\begin{proof}
We will ultimately apply the Rochberg--Semmes \cite{RS} notion of   NWO sequences, namely the  inequality \eqref{e:NWO}.   
By \eqref{tttt1}, we have
\begin{align}\label{comcom1}
(2n+1)^{(2n+2)k}\int_{\HH^{n}}|E_{k+1}(b)(g)-E_{k}(b)(g)|^{p}dg
&=\sum_{T\in\tile_{-k}}\fint_{T}|E_{k+1}(b)(g)-E_{k}(b)(g)|^{p}dg\nonumber\\
&\leq C\sum_{T\in\tile_{-k}}|T|^{-p/2}\left|\int_{T}b(g)h_{T}(g)dg\right|^{p}.
\end{align}
To continue, for any $T\in\tile_{-k}$, let $\hat{T}$ be the tile chosen in Theorem \ref{nondegen} with $N=0$, then $K_{\ell}({\hat g}^{-1}g)$ does not change sign for all $(g,{\hat g})\in T\times\hat{T}$ and
\begin{align}\label{lower}
|K_{\ell}({\hat g}^{-1}g)|\geq\frac{C}{|T|},
\end{align}
for some constant $C>0$.
 Also, let $\alpha_{\hat{T}}(b)$ be a median value of $b$ over $\hat{T}$. This  means $\alpha_{\hat{T}}(b)$ is a real number such that defining 
for a tile $ S$, 
\begin{align} \label{e:E1S}
E_{1}^{S}:=\left\{g\in S:b(g) < \alpha_{\hat{T}}(b)\right\}\ \ {\rm and}\ \
E_{2}^{S}:=\left\{g\in S:b(g)>\alpha_{\hat{T}}(b)\right\}, 
\end{align}
we have, with $ S= \hat T $,   the upper bound $  \lvert  E ^{\hat T} _{j}\rvert \leq \tfrac{1}2 \lvert  \hat  T\rvert  $ for $ j=1,2$. 
A median value always exists, but  may not be unique (see for example \cite{Journe}).
We use the notation $ $

Next we decompose $T$ into a union of sub-tiles by writing $T=\bigcup_{i=1}^{M_{n}}P_{i}$, where $P_{i}\in\tile_{-k-1}$ and $P_{i}\subseteq T$ satisfying $P_{i}\neq P_{j}$ if $i\neq j$.
By the cancellation property of $h_{T}$, we see that
\begin{align}\label{comcom2}
|T|^{-1/2}\left|\int_{T}b(g)h_{T}(g)dg\right|&=|T|^{-1/2}\left|\int_{T}(b(g)-\alpha_{\hat{T}}(b))h_{T}(g)\,dg\right|\nonumber\\
&\leq \frac{1}{|T|}\int_{T}\left|b(g)-\alpha_{\hat{T}}(b)\right|dg\nonumber\\
&\leq \frac{1}{|T|}\sum_{i=1}^{M_n}\int_{P_{i}}\left|b(g)-\alpha_{\hat{T}}(b)\right|dg\nonumber\\
&\leq \frac{1}{|T|}\sum_{i=1}^{M_n}\int_{P_{i}\cap E_{1}^{T}}\left|b(g)-\alpha_{\hat{T}}(b)\right|dg+ \frac{1}{|T|}\sum_{i=1}^{M_n}\int_{P_{i}\cap E_{2}^{T}}\left|b(g)-\alpha_{\hat{T}}(b)\right|dg\nonumber\\
&=:{\rm I}_{1}^{T}+{\rm I}_{2}^{T}.
\end{align}
Above, we are using the notation \eqref{e:E1S}.  

Now we denote
\begin{align*}
F_{1}^{T}:=\{{\hat g}\in \hat{T}:b({\hat g})\geq\alpha_{\hat{T}}(b)\}\ \ {\rm and}\ \
F_{2}^{T}:=\{{\hat g}\in \hat{T}:b({\hat g})\leq\alpha_{\hat{T}}(b)\}.
\end{align*}
Then by the definition of $\alpha_{\hat{T}}(b)$, we have $|F_{1}^{T}|=|F_{2}^{T}|\sim|\hat{T}|$ and $F_{1}^{T}\cup F_{2}^{T}=\hat{T}$. Note that for $s=1,2$, if $g\in E_{s}^{T}$ and $ \hat g\in F_{s}^{T}$, then 
\begin{align*}
\left|b(g)-\alpha_{\hat{T}}(b)\right|&\leq\left|b(g)-\alpha_{\hat{T}}(b)\right|+\left|\alpha_{\hat{T}}(b)-b({\hat g})\right|\\
&=\left|b(g)-\alpha_{\hat{T}}(b)+\alpha_{\hat{T}}(b)-b({\hat g})\right|= \left|b({\hat g})-b(g)\right|.
\end{align*}
Therefore, for $ s=1,2$, 
\begin{align}\label{haha}
{\rm I}_{s}^{T}&\lesssim \frac{1}{|T|}\sum_{i=1}^{M_n}\int_{P_{i}\cap E_{s}^{T}}\left|b(g)-\alpha_{\hat{T}}(b)\right|dg\frac{|F_{s}^{T}|}{|T|}\nonumber\\
&\lesssim \frac{1}{|T|}\sum_{i=1}^{M_n}\int_{P_{i}\cap E_{s}^{T}}\int_{F_{s}^{T}}\left|b(g)-\alpha_{\hat{T}}(b)\right|\left|K_{\ell}({\hat g}^{-1}g)\right|d{\hat g}dg\nonumber\\
&\lesssim \frac{1}{|T|}\sum_{i=1}^{M_n}\int_{P_{i}\cap E_{s}^{T}}\int_{F_{s}^{T}}\left|b({\hat g})-b(g)\right|\left|K_{\ell}({\hat g}^{-1}g)\right|d{\hat g}dg\nonumber\\
&=\frac{1}{|T|}\sum_{i=1}^{M_n}\left|\int_{P_{i}\cap E_{s}^{T}}\int_{F_{s}^{T}}(b({\hat g})-b(g))K_{\ell}({\hat g}^{-1}{\hat g})d{\hat g}dg\right|,
\end{align}
where in the last equality we used the fact that $K_{\ell}({\hat g}^{-1}g)$ and $b({\hat g})-b(g)$ do not  change sign for $(g,{\hat g})\in (T_{i}\cap E_{s}^{T})\times F_{s}^{T}$, $s=1,2$. This, in combination with the inequalities \eqref{comcom1} and \eqref{comcom2}, implies that
\begin{align}\label{eee ortho S norm} 
(2n+1)^{(2n+2)k}\int_{\HH^{n}} &|E_{k+1}(b)(g)-E_{k}(b)(g)|^{p}dg 
\\
& \lesssim \sum_{T\in\tile_{-k}}|T|^{-p/2}\left|\int_{T}b(g)h_{T}(g)dg\right|^{p}
\\
& \lesssim  \sum_{s=1}^{2}\sum_{T\in\tile_{-k}}\left|{\rm I}_{s}^{T}\right|^{p}
\\
&\lesssim  \sum_{s=1}^{2}\sum_{T\in\tile_{-k}}\left(\sum_{i=1}^{M_n}\left|\left\langle[b,R_{\ell}] \frac{|P_{i}|^{1/2}
\chi_{F_{s}^{T}}}{|T|},\frac{\chi_{E_{s}^{T}}}{|P_{i}|^{1/2}}\right\rangle\right|\right)^{p}.
\end{align}
Note that $e_T:=\frac{|P_{i}|^{1/2}
\chi_{F_{s}^{T}}}{|T|} \subset \hat T$ and $f_T:=\frac{\chi_{E_{s}^{T}}}{|P_{i}|^{1/2}} \subset T$. 
Based on Theorem \ref{nondegen} with $N=0$, we see that for each $T\in \tile_{-k}$, there is a unique $T_{A_0}\in \tile_{-k+A_0}$ such that $T, \hat T\subset T_{A_0}$. Hence,  $|e_T|, |f_T|\leq C|T_{A_0}|^{-{1\over2}}\chi_{T_{A_0}} $, where $C$ is an absolute constant depending only on $n$ and $A_0$. Note also that each $T_{A_0}\in \tile_{-k+A_0}$ contains only a finite number (depending on $n,A_0$) of $T\in \tile_{-k}$ with $T, \hat T\subset T_{A_0}$.
Sum this last inequality over $ k\in \mathbb Z $, and appeal  to \eqref{e:NWO} to conclude the  Lemma.  
\end{proof}

This is an immediate corollary. 

\begin{coro}\label{p:step1}
Let $2n+2<p<\infty$ and suppose that $b$ is a locally integrable function satisfying $\|[b,R_{\ell}]\|_{S^p}<\infty$, then there exists a constant $C>0$ such that for any $k\in\mathbb{Z}$,
\begin{align} \label{e:step1}
\|b-E_{k}(b)\|_{p}\leq C(2n+1)^{-(2n+2)k/p}\|[b,R_{\ell}]\|_{S^p}.
\end{align}
\end{coro}

\begin{proof}
 Note that $E_{k}(b)\rightarrow b$ a.e. as $k\rightarrow \infty$. Since $ p > 2n+2$,  it suffices to show that $\|E_{k+1}(b)-E_{k}(b)\|_{p}\leq C(2n+1)^{-(2n+2)k/p}\|[b,R_{\ell}]\|_{S^{p}}$.  But that is a consequence of Lemma~\ref{step1}.
\end{proof}

Comparing the lemma below to Lemma~\ref{step1}, we are replacing $ E _{k+1}b$ with $ b$, and hence we require $ p$ to be strictly bigger than   
$ 2n+2$.  

\begin{lemma}\label{step2}
Let $2n+2<p<\infty$ and suppose that $b\in L_{{\rm loc}}^{1}(\HH^{n})$, then 
\begin{align}\label{eee Besov type}
\left(\sum_k (2n+1)^{(2n+2)k}\|b-E_{k}(b)\|_{p}^{p}\right)^{1/p}&\lesssim  \|[b,R_{\ell}]\|_{S^{p}}.
\end{align}
\end{lemma}
\begin{proof}
Denote the left-hand side of \eqref{eee Besov type} by $ \mathfrak J$. 
Then we see that
\begin{align*}
\mathfrak J&\leq \left(\sum_k(2n+1)^{(2n+2)k}\|b-E_{k+1}(b)\|_{p}^{p}\right)^{1/p}+\left(\sum_k(2n+1)^{(2n+2)k}\|E_{k+1}(b)-E_{k}(b)\|_{p}^{p}\right)^{1/p}\\
&=\left(\sum_{k}(2n+1)^{(2n+2)(k-1)}\|b-E_{k}(b)\|_{p}^{p}\right)^{1/p}+\left(\sum_k(2n+1)^{(2n+2)k}\|E_{k+1}(b)-E_{k}(b)\|_{p}^{p}\right)^{1/p}\\
&\leq 2^{-(2n+2)/p}\left(\sum_k(2n+1)^{(2n+2)k}\|b-E_{k}(b)\|_{p}^{p}\right)^{1/p}
+\left(\sum_k(2n+2)^{(2n+2)k}\|E_{k+1}(b)-E_{k}(b)\|_{p}^{p}\right)^{1/p}\\
&=: {\textup{Term}_{1}}+{\textup{Term}_{2}}. 
\end{align*}
Since $ 2n+2 < p < \infty $,  we see that $ {\textup{Term}_{1}}$ can be absorbed into $ \mathfrak J$.  
  Lemma~\ref{step1} controls $ {\textup{Term}_{2}}$. 
\end{proof}

\begin{proposition}\label{schattenlarge1}
Let $2n+2<p<\infty$ and suppose that $b\in L_{{\rm loc}}^{1}(\HH^{n})$, then there exists a constant $C>0$ such that
\begin{align*}
\|b\|_{B_{p,p}^{\frac{2n+2}{p}}(\HH^n)}\leq C\|[b,R_{\ell}]\|_{S^p}.
\end{align*}
\end{proposition}
\begin{proof}
To begin with, we note that
\begin{align}
\int_{\HH^{n}}\int_{\HH^{n}}\frac{|b(g)-b({\hat g})|^{p}}{d(g,{\hat g})^{2(2n+2)}}dgd\hat g\lesssim \sum_{k\in\mathbb{Z}}(2n+1)^{2(2n+2)k}\iint_{d(g,{\hat g})\leq (2n+1)^{-k-1}}|b(g)-b({\hat g})|^{p}dgd{\hat g}.
\end{align}
Hence, it suffices to show that
\begin{align}\label{maingoal}
\sum_{k=L}^{M}(2n+1)^{2(2n+2)k}\iint_{d(g,\hat g)\leq (2n+1)^{-k-1}}|b(g)-b({\hat g})|^{p}dgd{\hat g}\leq C\|[b,R_{\ell}]\|_{S^{p}}^p,
\end{align}
where $C$ is a constant independent of $L<M\in\mathbb{Z}$.

Recall that a tile $T$ in $\tile_{-k}$ is approximately a Heisenberg ball of radius $(2n+1)^{-k}$. Fix a Heisenberg ball $B$ centered at the origin with radius $(2n+1)^{-L+A}$ for a large fixed integer $ A$, and then denote $b_{\tilde g}(g) := b(\tilde g g)$ for $\tilde g\in \HH^n$. Then
 the left-hand side of \eqref{maingoal} is dominated by a constant times 
\begin{align}\label{eee prob average}
{1\over |B|}\int_B \sum_{k=L}^{M}\sum_{T\in\tile_{-k}} & (2n+1)^{2(2n+2)k} \int_T\int_T |b_{\tilde g}(g)-b_{\tilde g}({\hat g})|^p\, d{\hat g}\, dg\, d\tilde g\nonumber\\
&\lesssim {1\over |B|}\int_B \sum_{k=L}^{M}\sum_{T\in\tile_{-k}}  (2n+1)^{(2n+2)k} \int_T| b_{\tilde g}(g)- E_k(b_{\tilde g})(g) |^p\,dg\,d\tilde g\nonumber\\
&\lesssim {1\over |B|}\int_B C\|[b_{\tilde g},R_{\ell}]\|_{S^{p}}^p\, d \tilde g,
\end{align}
where in the first inequality we added and subtracted the term $E_k(b_{\tilde g})$ by noting that for $g,\hat g\in T\in\tile_{-k}$, $E_k(b_{\tilde g})(g)=E_k(b_{\tilde g})(\hat g)$, and in the second inequality we use Lemma \ref{step2}. Next,  as the Riesz transform is convolution,  $\|[b_{\tilde g},R_{\ell}]\|_{S^{p}}=\|[b,R_{\ell}]\|_{S^{p}}$, we obtain that  the right-hand side of \eqref{eee prob average} is bounded by
$C\|[b,R_{\ell}]\|_{S^{p}}^p$. Hence, \eqref{maingoal} holds.

Therefore, the proof of Proposition \ref{schattenlarge1} is complete.
\end{proof}

\subsection{Proof of the sufficient condition}
\label{s:Sufficient}

\begin{proposition}\label{schattenlarge2}
Suppose $\ell\in \{1,2,\cdots,2n\}$, $2n+2<p<\infty$ and $b\in  L^1_{{\rm loc}}(\HH^n)$. If  $b\in B_{p,p}^{\frac{2n+2}{p}}(\HH^n)$, then $[b,R_{\ell}]\in S^p$.
\end{proposition}
\begin{proof}
  We follow the proof in \cite{JW},  which relies upon general estimates for Schatten norms of integral operators.   
  For the convenience of the readers, we briefly sketch the proof here. We first recall that $[b,R_{\ell}]$ is compact \cite{CDLW} when $b\in B_{p,p}^{\frac{2n+2}{p}}(\HH^n)\subset {\rm VMO}(\HH^n)$. Note that Russo (\cite{Russo}) proved that for general measure space $(X,\mu)$, if $p>2$ and $K(x,y)\in  L^{2}(X\times X)$, then the integral operator $T$ associated to the kernel $K(x,y)$ satisfies the following bound:
\begin{align*}
\|T\|_{S^{p}}\leq \|K\|_{L^p,L^{p^{\prime}}}^{1/2}\|K^{*}\|_{L^p,L^{p^{\prime}}}^{1/2},
\end{align*}
where $p'$ is the conjugate index of $p$, $K^{*}(x,y)=\overline{K(y,x)}$, and $\|\cdot\|_{L^p, L^{p^{\prime}}}$ denotes the mixed-norm:
$
\|K\|_{L^p,L^{p^{\prime}}}:=\big\|\|K(x,y)\|_{L^p(dx)}\big\|_{L^{p^{\prime}}(dy)}.
$
Later on Goffeng (\cite{Goffeng}) showed that the condition $K(x,y)\in  L^{2}(X\times X)$ in the above statement can be removed.

Moreover, Janson--Wolff (\cite[Lemma 1 and Lemma 2]{JW}) extended the above statement to the corresponding weak-type version general measure space $(X,\mu)$: if $p>2$ and $1/p+1/p^{\prime}=1$, then
\begin{align}\label{integral}
\|T\|_{S^{p,\infty}}\leq \|K\|_{L^{p},L^{p^{\prime},\infty}}^{1/2}\|K^{*}\|_{L^{p},L^{p^{\prime},\infty}}^{1/2},
\end{align}
where $\|\cdot\|_{L^p, L^{p^{\prime},\infty}}$ denotes the mixed-norm:
$
\|K\|_{L^p,L^{p^{\prime},\infty}}:=\big\|\|K(x,y)\|_{L^p(dx)}\big\|_{L^{p^{\prime},\infty}(dy)}.
$

Next, back to our setting on Heisenberg group, we note that  by weak-type Young's inequality, for $1/q=1-2/p$,
\begin{align}\label{verify1}
\left\|(b(g)-b(\hat g))K(g,\hat g)\right\|_{L^p, L^{p^{\prime},\infty}}
&\leq
\left\|\frac{b(g)-b(\hat g)}{d(g,\hat g)^{2n+2}}\right\|_{L^p, L^{p^{\prime},\infty}}\nonumber\\
&\leq \left\|\frac{b(g)-b(\hat g)}{d(g,\hat g)^{2(2n+2)/p}}\right\|_{L^{p},L^{p}}\left\|\frac{1}{d(g,\hat g)^{(2n+2)(1-2/p)}}\right\|_{L^{\infty},L^{q,\infty}}\\
&\leq C\|b\|_{B_{p,p}^{(2n+2)/p}(\HH^n)}.\nonumber
\end{align}
Similarly,
\begin{align}\label{verify2}\left\|(b(g)-b(\hat g))\overline{K(\hat g,g)}\right\|_{L^p, L^{p^{\prime},\infty}}\leq C\|b\|_{B_{p,p}^{(2n+2)/p}(\HH^n)},
\end{align}
Combining the inequalities \eqref{verify1}, \eqref{verify2} and then applying the weak-type Russo's inequality \eqref{integral}, we see that $$\|[b,R_{\ell}]\|_{S^{p,\infty}}\leq C\|b\|_{B_{p,p}^{(2n+2)/p}(\HH^n)}.$$ Since this inequality holds for all $2n+2<p<\infty$, we can apply the interpolation $(S^{p_1},S^{p_2})_{\theta_p}=S^{p}$ and $(B_{p_1,p_1}^{(2n+2)/p_1},B_{p_2,p_2}^{(2n+2)/p_2})_{\theta_{p}}=B_{p,p}^{(2n+2)/p}$, where $\frac{1-\theta_p}{p_1}+\frac{\theta_p}{p_2}=\frac{1}{p}$, to obtain that
\begin{align*}
\|[b,R_{\ell}]\|_{S^{p}}\leq C\|b\|_{B_{p,p}^{(2n+2)/p}(\HH^n)}.
\end{align*}
This finishes the proof of sufficient condition for the case $2n+2<p<\infty$.
\end{proof}

\section{Theorem \ref{schatten}:  $0<p\leq 2n+2$}\label{four}
\setcounter{equation}{0}

In this section, we prove the  second argument in Theorem \ref{schatten}. That is, for each $\ell\in\{1,2,\ldots,2n\}$ and for $0<p\leq 2n+2$, the commutator $[b,R_{\ell}]$ is in $S^p$ if and only if
 $b$ is a constant. The sufficient condition is obvious, since $[b,R_{\ell}]=0$ when $b$ is a constant. Thus, it suffices to show the necessary condition. 
 It suffices to consider  the critical case $p=2n+2$, by the inclusion $ S^p\subset S^{q}$ for $p<q$.  

To formulate our argument simplicity, we will usually identity $\mathbb{C}^{n}$ with $\mathbb{R}^{2n}$ in Lemmas \ref{signlemma}---\ref{lowerbound} and use the following notation to denote the points of $\mathbb{C}^{n}\times \mathbb{R}\equiv \mathbb{R}^{2n+1}: g=[z,t]\equiv [x,y,t]=[x_{1},\cdots,x_{n},y_{1},\cdots,y_{n},t]$ with $z=[z_{1},\cdots,z_{n}]$, $z_{j}=x_{j}+iy_{j}$ and $x_{j},y_{j},t\in\mathbb{R}$ for $j=1,\cdots,n$. Then the multiplication law can be explicitly expressed as
\begin{align*}
gg^{\prime}=[x,y,t][x^{\prime},y^{\prime},t^{\prime}]=[x+x^{\prime},y+y^{\prime},t+t^{\prime}+2\langle y,x^{\prime}\rangle-2\langle x,y^{\prime} \rangle],
\end{align*}
where $\langle \cdot,\cdot\rangle$ denotes the standard inner product in $\mathbb{R}^{n}$.

\begin{lemma}\label{signlemma}
There exists a positive integer $B_0$ such that for any tile $T\in \tile_{-k}$ and $a_{j}=\pm 1$ ($j=1,2,\cdots,2n$), there are tiles $T^{\prime}\in\tile_{-k-B_0}$, $T^{\prime\prime}\in\tile_{-k-B_0}$ such that $T^{\prime}\subset T$,
$T^{\prime\prime}\subset T$ and if $g=(g_{1},\cdots,g_{2n},t)\in T^{\prime\prime}$, $h=(h_{1},\cdots,h_{2n},t^{\prime})\in T^{\prime}$, then $a_{j}(g_{j}-h_{j})\gtrsim \wid(T)$ $(j=1,2,\ldots,2n)$.
\end{lemma}
\begin{proof}
Consider first $T = \delta_{(2\cdim+1)^k} (T_o)$.  Based on (4) in Lemma \ref{thm:Heisenberg-grid}, we see that $B(o, C_1(2\cdim+1)^k)\subset T$. Then one can choose $g_{o,1}\in B(o, C_1(2\cdim+1)^k)$ such that $d(g_{o,1},o) = {3C_1\over4}(2\cdim+1)^k$, and that all the first $2n$ components of $g_{o,1}$ is positive and equals to ${3C_1\over4}(2\cdim+1)^k$.
Thus, we have $B(g_{o,1}, {C_1\over 40} (2\cdim+1)^k) \subset B(o, C_1(2\cdim+1)^k)$ and that for every $x=(x_1,\ldots,x_{2n},t_x)\in B(g_{o,1}, {C_1\over 40} (2\cdim+1)^k),$ we have $x_i>0$ and is equivalent to ${3C_1\over4}(2\cdim+1)^k$. Then taking the inverse of the ball $B(g_{o,1}, {C_1\over 40} (2\cdim+1)^k)$, we get other ball $B(g_{o,2}, {C_1\over 40} (2\cdim+1)^k)$ such that $g_{o,2}=g_{o,1}^{-1}$ and that for every $y=(y_1,\ldots,y_{2n},t_y)\in B(g_{o,2}, {C_1\over 20} (2\cdim+1)^k),$ we have $y_i<0$ and is equivalent to $-{3C_1\over4}(2\cdim+1)^k$. As a consequence, we see that there exist $T^{\prime}\in\tile_{-k-B_0}$ such that $T'\subset B(g_{o,1}, {C_1\over 40} (2\cdim+1)^k)$ and $T^{\prime\prime}\in\tile_{-k-B_0}$ such that
$T''\subset B(g_{o,2}, {C_1\over 20} (2\cdim+1)^k)$. Then it is clear that if $g\in T^{\prime\prime}$, $h\in T^{\prime}$, then $g_{j}-h_{j}\gtrsim \wid(T)$ $(j=1,2,\ldots,2n)$.

For general $T\in \tile_{-k}$ with $u=\cent(T)$, we know that $T = \delta_{(2\cdim+1)^k} (u) \cdot \delta_{(2\cdim+1)^k} (T_o)$. Hence, the argument holds by using the translation and dilation. This ends the proof of Lemma \ref{signlemma}.
\end{proof}

Recall the following first order Taylor's inequality on Heisenberg group from \cite{Taylorformula}.
\begin{lemma}\label{taylor}
Let $f\in C^{\infty}(\HH^{n})$, then for every $g=(x_{1},\cdots,x_{2n},t),g_{0}=(x_{0}^{1},\ldots,x_{0}^{2n},t_0)\in \HH^{n}$, we have
\begin{align*}
f(g)=f(g_{0})+\sum_{k=1}^{2n}\frac{X_{k}f(g_{0})}{k!}(x_{k}-x_{0}^{k})+R(g,g_{0}),
\end{align*}
where the remainder  $R(g,g_{0})$ satisfies the following inequality:
\begin{align*}
|R(g,g_{0})|\leq C\left(\sum_{k=1}^{2}\frac{c^{k}}{k!}\sum_{\substack{i_{1},\ldots,i_{k}\leq 2n+1,\\ I=(i_{1},\ldots,i_{k}),\ d(I)\geq 2}}\rho(g_{0}^{-1}g)^{d(I)}\sup\limits_{\rho(z)\leq c\rho(g_{0}^{-1}g)}|X^{I}f(g_{0}z)|\right)
\end{align*}
for some constant $c>0$.
\end{lemma}

We denote $\nabla$ be the horizontal gradient of $\HH^{n}$ defined by $\nabla f:=(X_{1}f,\cdots,X_{2n}f)$. Then we can show a lower bound for a local pseudo-oscillation of the symbol $b$ in the commutator.

\begin{lemma}\label{lowerbound}
Let $b\in C^{\infty}(\HH^n)$. Assume that there is a point $g_{0}\in\HH^{n}$ such that $\nabla b(g_{0})\neq 0$. Then there exist $C>0$, $\varepsilon>0$ and $N>0$ such that if $k>N$, then for any tile $T\in \tile_{-k}$ satisfying $d(\cent(T),g_{0})<\varepsilon$, one has
\begin{align}  \label{e:lowerbound}
\Bigl\lvert 
\fint _{T'} b 
-  
\fint _{T''} b 
\Bigr\rvert  \geq C \wid(T)|\nabla b(g_{0})|.
\end{align}
Above,  $T^{\prime}$ and $ T''$ are the tiles chosen in Lemma \ref{signlemma}.
\end{lemma}
\begin{proof}
Denote $c_{T}:=\cent(T):=\{c_{T}^{1},\ldots,c_{T}^{2n},t_T\}$ and $g=(g_1,\cdots,g_{2n},t)$, then by Lemma \ref{taylor},
\begin{equation}\label{e:b}
b(g)=b(c_{T})+\sum_{j=1}^{2n}\frac{X_{j}b(c_{T})}{j!}(g_{j}-c_{T}^{j})+R(g,c_{T}), 
\end{equation}
where the remainder term $R(g,c_{T})$ satisfies
\begin{align*}
|R(g,c_{T})|\leq C\left(\sum_{j=1}^{2}\frac{c^{j}}{j!}\sum_{\substack{i_{1},\ldots,i_{j}\leq 2n+1,\\ I=(i_{1},\ldots,i_{j}),\ d(I)\geq 2}}\rho(c_{T}^{-1}g)^{d(I)}\sup\limits_{\rho(z)\leq c\rho(c_{T}^{-1}g)}|X^{I}b(c_{T}z)|\right).
\end{align*}
Note that the condition $\rho(z)\leq c\rho(c_{T}^{-1}g)$ implies that $d(c_{T}z,c_{T})=\rho(z)\leq c\rho(c_{T}^{-1}g)\lesssim \wid(T)$ whenever $g\in T$. Hence, if $g\in T$, then
\begin{align*}
|R(g,c_{T})| \lesssim \wid(T)^{2}\sum_{j=1}^{2}\sum_{\substack{i_{1},\ldots,i_{j}\leq 2n+1,\\ I=(i_{1},\ldots,i_{j}),\ d(I)\geq 2}}\|X^{I}b\|_{L^{\infty}(B(g_{0},1))}.
\end{align*}
For $ \epsilon = \epsilon _{b} >0$ sufficiently small, this last estimate is smaller than the right hand side of \eqref{e:lowerbound}.   
That is, in \eqref{e:b}, we are only concerned with the first two terms on the right.

Apply Lemma \ref{signlemma}, with the choice of signs $ a_j = {\rm sgn}(X_{j}b)(c_{T})$.  
Let $T' , T''$ be the tiles that this Lemma provides to us.     
For $  g'\in (g_j')  $ and $ g'' = (g''_j) \in T''$, we have 
\begin{equation*}
{\rm sgn}(X_{j}b)(c_{T})(g_{j}-h_{j})\gtrsim \wid(T), \qquad j=1 ,\dotsc, 2n. 
\end{equation*}
Therefore, we can estimate 
\begin{align}  \label{e:TT}
\Bigl\lvert &
\fint _{T'} b (g') \; dg'  
-  
\fint _{T'} b  (g'')\; dg''  
\Bigr\rvert
\\
& \geq 
c  
\Bigl\lvert 
\fint _{T'} \fint _{T''}   \sum_{j=1}^{2n}\frac{(X_{j}b)(c_{T})}{j!}(g_{j}'-g''_{j})  \;dg' dg'' 
\Bigr\rvert 
-\fint _{T'} \lvert  R(g',c_{T})\rvert  \; dg'     
- \fint _{T''} \lvert  R(g'',c_{T})\rvert  \; dg''     
\\
& \geq 
c \sum_{j=1}^{2n}|X_{j}b(c_{T})|\wid(T)-C\wid(T)^{2}\sum_{j=1}^{2}\sum_{\substack{i_{1},\ldots,i_{j}\leq 2n+1,\\ I=(i_{1},\ldots,i_{j}),\ d(I)\geq 2}}\|X^{I}b\|_{L^{\infty}(B(g_{0},1))}\\
& \gtrsim  C\wid(T)|\nabla b(g_{0})|.  
\end{align}
This inequality completes the Lemma. 
%
\end{proof}

\begin{lemma}\label{const}
A function $b\in L_{{\rm loc}}^{1}(\HH^{n})$ is constant if  
\begin{align} \label{e:Bsup}
\sup\limits_{h\in B(o,1)}\left\| \left\{ \fint_{T}\fint _T 
\left|  E _{k + B_0}  ( \tau ^{h}b)(g')-  E _{k+B_0} (\tau ^{h}b) (g'')    \right| dg'dg''\right\}_{T\in\tile}\right\|_{\ell^{2n+2}}<+\infty. 
\end{align}
(In the display, $ T\in \tile_k$, and both $ T$ and $ k$ vary. And $ \tau ^{h}$ denotes translation by $ h$.)   
\end{lemma}
\begin{proof}
The assumption is that $ b\in L_{{\rm loc}}^{1}(\HH^{n})$, but the previous Lemmas require $ b$ to be smooth.  
Denote $\psi_{\epsilon}(g):=\epsilon^{-2n-2}\psi(\delta_{\epsilon^{-1}}g)$, where $\psi$ is a smooth  compactly supported bump function which integrates to zero, and $\epsilon$ is a small positive constant. 
Then, $b _{\epsilon } = b\ast \psi_{\epsilon}$ is smooth.  We argue that these are all constant. And, they converge to $ b$ pointwise so this is sufficient.  

The point is that $ b _{\epsilon }$ is smooth, and that with the supremum on the outside in \eqref{e:Bsup}, we have 
\begin{equation}  \label{e:norm}
 \left\| \left\{ \fint_{T}\fint _T 
\left|  E _{k + B_0}  (b _{\epsilon })(g')-  E _{k+B_0}(b _{\epsilon })(g'')    \right| dg'dg''\right\}_{T\in\tile}\right\|_{\ell^{2n+2}} < \infty .  
\end{equation}


If $ b _{\epsilon }$ is not constant, we argue that the norm above is actually infinite, which is a contradiction.  
It follows from \cite[Proposition 1.5.6]{Liebook} that there exists a point $g_{0}\in \HH^{n}$ such that $\nabla b\ast\psi_{\epsilon}(g_{0})\neq 0$. But then,  Lemma \ref{lowerbound} applies.  
There exist  $\varepsilon>0$ and $N>0$ such that if $k>N$, then for any tile $T\in \tile_{-k}$ satisfying $d(\cent(T),g_{0})<\varepsilon$,
\begin{align*}
\fint_{T}\fint _T 
\left|  E _{k + B_0}  (b _{\epsilon })(g')-  E _{k+B_0}(b _{\epsilon })(g'')    \right| dg'dg''  \gtrsim \wid (T).  
\end{align*}
Note that for $k>N$, the number of $T\in\tile_{-k}$ and $d(\cent(T),g_{0})<\varepsilon$ is at least
\begin{equation*}
c(2n+1)^{k(2n+2)} \simeq \wid (T) ^{- (2n+2)}.  
\end{equation*}
But then, it is clear that the norm in \eqref{e:norm} is infinite. 
\end{proof}

\begin{proposition}\label{mainprop}
Suppose  $b\in  L^1_{{\rm loc}}(\HH^n)$ and $p=2n+2$. Then
for any $\ell\in \{1,2,\cdots,2n\}$, the commutator $[b,R_{\ell}]\in S^{2n+2}$ if and only if $b$ is a constant.
\end{proposition}

\begin{proof}
A constant function $ b$ is associated with the zero commutator. So, we only consider the 
direction in which we assume $[b,R_{\ell}]\in S^ {2n+2}$.   And, then, we need to verify that \eqref{e:Bsup} holds.  
That inequality has the supremum over translations.  The Riesz transforms are themselves convolution operators, 
so that it suffices to verify \eqref{e:Bsup} without translations.  That is, 
\begin{equation}   \label{e:intint}
\left\| \left\{ \fint_{T}\fint _T 
\left|  E _{k + B_0} (b)(g')-  E _{k+B_0}(b)(g'')    \right| dg'dg''\right\}_{T\in\tile}\right\|_{\ell^{2n+2}} 
\lesssim \lVert [ b, R _{\ell}]\rVert _{S ^{2n+2}} < \infty .  
\end{equation}
(In the display, $ T\in \tile_k$, and both $ T$ and $ k$ vary).

This is in fact a corollary to Lemma~\ref{step1}, and is seen by way of a general remark. 
For a random variable $ Z$, we have for $ 1\leq p < \infty $, 
\begin{equation*}
 \lVert Z - \mathbb E Z  \rVert_p  \simeq  \lVert Z - Z'\rVert_p, 
\end{equation*}
where $ Z'$ is an independent copy of $ Z$.  Indeed, 
\begin{align*}
 \lVert Z - \mathbb E Z  \rVert_p  & =  \lVert Z - \mathbb E Z ' \rVert_p 
 \\
 & \leq   \lVert Z - Z'\rVert_p  \leq 2  \lVert Z - \mathbb E Z  \rVert_p. 
\end{align*}
The first inequality is by convexity and the second by the triangle inequality.    

Thus, Lemma~\ref{step1} implies 
\begin{equation*}
\left\| \left\{ \fint_{T} 
\left|  E _{k + B_0} (b)(g)-  E _{k} (b)(g)     \right|\,dg \right\}_{T\in\tile}\right\|_{\ell^{2n+2}} 
\lesssim \lVert [ b, R _{\ell}]\rVert _{S ^{2n+2}} 
\end{equation*}
as $ B_0$ is a fixed integer.  And then \eqref{e:intint} follows.  
\end{proof}

\section{Applications}\label{appl}
\setcounter{equation}{0}

As stated in the introduction, our approach depends upon a standard non-degeneracy condition on the kernel of the 
singular integral operator, and then on robust real variable techniques.  (In particular, no Fourier analysis.) 
The approach applies to the following non-Euclidean Calder\'on--Zygmund operators. 

\smallskip 

(1)  The Cauchy--Szeg\H{o} projection $\mathcal{C}$ \cite[Chapter 12, Section 2.4]{St} is an important singular integral on $\HH^n$. 
It recovers   an analytic function in the Siegel upper half space  from its boundary value.
Its restriction to the boundary is a convolution operator, that is, $\mathcal{C}(f)(g)=\int_{\HH^n} f (g') k_{CS}((g')^{-1}g)dg'$, and the convolution kernel $k_{CS}$ is given by
\begin{equation}\label{Cauchy Szego}
k_{CS}(g)= \frac{c}{(|z|^2+\i\, t)^{\cdim+1}},\quad \i^2=-1, 
\qquad\forall g=(z,t)\in \HH^n.
\end{equation}
It is well-known that this $k_{CS}$ is a Calder\'on--Zygmund kernel. From the explicit kernel, we see that the non-degeneracy condition in our Theorem \ref{nondegen} holds for $\mathcal{C}$.  Hence,  Theorem \ref{schatten} holds for $[b, \mathcal{C}]$. This recovers the Theorem A obtained by Feldman--Rochberg \cite{FR} where they relied on the Cayley transform and Fourier transform.

\smallskip 

(2) Second order Riesz transforms appear naturally in the study of PDEs (see for instance \cite{GT}) and have been extensively studied in literature. They are mostly interpreted as iterations of Riesz transforms and their adjoints, or second derivatives of the fundamental solution operator for the Laplacian: $\partial_i\partial_j (-\Delta)^{-1}$. On Euclidean spaces, second order Riesz transforms are well understood as Calder\'on--Zygmund singular integrals and have bounded $L^p$ norm for $1<p<\infty$.

\smallskip
\ \  (2a)\ \ A particular interesting example is the classical Beurling--Ahlfors operator $\mathcal{B}$ on the complex plane defined by (see for example \cite{BW,PV})
%
%
\begin{align*}
\mathcal{B}(f)(z)&= {\rm p.v.} {1\over \pi} \int_{\C} {f(w) \over \big( z-w\big)^2} \ dw.
\end{align*}
Equivalently, we have $$\mathcal{B} = \partial^2 (-\Delta)^{-1}, $$
where $\partial= {\partial\over \partial x_1}- \i{\partial\over \partial x_2}$ is the Cauchy--Riemann operator and $\Delta$ is the Laplacian on $\mathbb R^2$.

Note that the kernel of $\mathcal B$ is  homogeneous and smooth away from the diagonal. Hence, the Schatten class $[b,\mathcal B]$ was covered by Rochberg--Semmes \cite{RS}. Our approach can also be applied to  $[b,\mathcal B]$, to have the explicit quantitative estimate for the Schatten norm.

\smallskip
\ \  (2b)\ \ Second order Riesz transform $\mathcal T (-\Delta_\HH)^{-1}$ on $\HH^n$ (recall that $\mathcal T = {1\over 4} (X_j X_{n+j} - X_{n+j} X_{j})$).

By using functional calculus for $(-\Delta_\HH)^{-1}$, it is direct to see that
$$ \mathcal T (-\Delta_\HH)^{-1}=  \int_0^\infty \mathcal T e^{h \Delta_\HH} \, dh,  $$
which gives that the kernel $K$ of $\mathcal T (-\Delta_\HH)^{-1}$ is a convolution kernel. Together with the size and smoothness estimates for the heat kernel \cite{VSC}, we obtain that for $g\not=[0,0]$,
$$ |K(g)|\lesssim {1\over \rho(g)^{2n+2}}\quad{\rm and}\quad  |X_\ell K(g)|\lesssim {1\over \rho(g)^{2n+3}} $$ for $\ell =1,2,...,2n$. Hence, $\mathcal T (-\Delta_\HH)^{-1}$ is a Calder\'on--Zygmund operator on $\HH^n$. We now verify the
non-degeneracy condition in our Theorem \ref{nondegen}.

We follow the idea in Section 7 in \cite{DLLWW}.
Recall that (\cite{G77}, see also Section 3.1) the explicit expression of heat kernel on the Heisenberg group $\mathbb H^n$ is
as follows: for $g=[z,t]\in\HH^n$,
\begin{align}\label{heat kernel ph}
p_h(g)={1\over 2(4\pi h)^{n+1}}\int_{\mathbb R}\exp\Big({\lambda\over 4h}\big(\i t-|z|_{\mathbb C^n}^2\coth\lambda\big)\Big)\Big({\lambda\over\sinh\lambda}\Big)^nd\lambda,\quad \i^2=-1.
\end{align}
For any $g=[z,t]\in\mathbb H^n$,  by using the explicit expression of the heat kernel above and by Fubini's theorem, we have that for $g\not=[0,0]$,
\begin{align*}
K(g)&={1\over 2(4\pi )^{n+1}}{\partial\over\partial t}\int_{0}^{+\infty}h^{-n-1}\int_{\mathbb R}\exp\Big({\lambda\over 4h}\big(\i t-|z|_{\mathbb C^n}^2\coth\lambda\big)\Big)\Big({\lambda\over\sinh\lambda}\Big)^n\ d\lambda\ dh\\
&={1\over 2(4\pi )^{n+1}}{\partial\over\partial t}\int_{\mathbb R}\,\int_{0}^{+\infty}h^{-n-1}\exp\Big({\lambda\over 4h}\big(\i t-|z|_{\mathbb C^n}^2\coth\lambda\big)\Big)dh\ \Big({\lambda\over\sinh\lambda}\Big)^nd\lambda\\
&=C_1{\partial\over\partial t}\int_{\mathbb R}\big(|z|_{\mathbb C^n}^2\lambda\coth\lambda-\i\lambda t\big)^{-n}\Big({\lambda\over\sinh\lambda}\Big)^nd\lambda\\
&=C_2  \int_{\mathbb R}\big(|z|_{\mathbb C^n}^2\lambda\coth\lambda-\i\lambda t\big)^{-n-1}\Big({\lambda\over\sinh\lambda}\Big)^n \lambda\ d\lambda,
\end{align*}
where in the next to the last equality we applied Cauchy integral formula to deform the ray on right-half complex plane $\mathbb{C}_{+}$ into the real axis. Here we also note that
\begin{align}\label{C1C2}
C_{2}=-n\, \i C_{1}=-\frac{n\, \i}{8\pi^{n+1}}\int_{0}^{\infty}s^{-n-1}e^{-s^{-1}}ds\neq 0.
\end{align}
Observe that
\begin{align*}
|z|_{\mathbb C^n}^2\lambda\coth\lambda-\i\lambda t
&={\lambda\over \sinh\lambda}d_K^2(g)\bigg({|z|_{\mathbb C^n}^2\over d_K^2(g)}\cosh\lambda-\i{t\over d_K^2(g)}\sinh\lambda\bigg)
={\lambda\over \sinh\lambda}d_K^2(g)\cosh(\lambda-\i\phi),
\end{align*}
where $d_K$ is the Kor\'anyi metric given by $d_K(g) = (|z|_{\mathbb C^n}^4+t^2)^{1\over 4}$ for $g=[z,t]\in \HH^n$, and
\begin{align}\label{phi}
-{\pi\over2}\leq\phi=\phi(|z|_{\mathbb C^n},t)\leq {\pi\over 2},\quad e^{\i\phi}=d_K^{-2}(g)(|z|_{\mathbb C^n}^2+\i\,t).
\end{align}
Thus, we have
\begin{align*}
K(g)
&=C_2  \int_{\mathbb R}\bigg( {\lambda\over \sinh\lambda}d_K^2(g)\cosh(\lambda-\i\phi) \bigg)^{-n-1}\Big({\lambda\over\sinh\lambda}\Big)^n \lambda\ d\lambda\\
&=C_2 d_K^{-2n-2}(g) \int_{\mathbb R}\big( \cosh(\lambda-\i\phi) \big)^{-n-1} \Big({\lambda\over \sinh\lambda}\Big)^{-n-1}\Big({\lambda\over\sinh\lambda}\Big)^n \lambda\ d\lambda\\
&=C_2 d_K^{-2n-2}(g) \int_{\mathbb R}\big( \cosh(\lambda-\i\phi) \big)^{-n-1} \ \sinh\lambda \ d\lambda\\
&=C_2 d_K^{-2n-2}(g) \int_{\mathbb R}\big( \cosh(\lambda) \big)^{-n-1} \ \sinh(\lambda+\i\phi) \ d\lambda,
\end{align*}
where the last equality follows from Cauchy integral formula again. We now define
$$
F(g):=\int_{\mathbb R}\big( \cosh(\lambda) \big)^{-n-1} \ \sinh(\lambda+\i\phi) \ d\lambda.
$$
Then we have $K(g)=C_2F(g) d_K^{-2n-2}(g) $. We now investigate the function
$$\mathfrak F(w):=\int_{\mathbb R}\big( \cosh(\lambda ) \big)^{-n-1} \ \sinh(\lambda+w) \ d\lambda, \quad w\in\mathbb C.$$
Then we have $F(g) = \mathfrak F(\i \phi)$  with $g=[z,t]\not=0$ and $\phi=\phi(|z|_{\mathbb C^n}, t)$ such that $e^{\i\phi}=d_K^{-2}(g)(|z|_{\mathbb C^n}^2+\i\,t)$. Note that
$\mathfrak F(w)$ is analytic in some domain in the complex plane $\mathbb C$, which contains the line segment $[-{\pi \over2}\i,{\pi \over2}\i]$ in the imaginary axis, and that $\mathfrak F({\pi \over4}\i)\not=0$.

Thus,
$\mathfrak F(w)$ has at most a finite number of zero points on $[-{\pi \over2}\i,{\pi \over2}\i]$, i.e., there exist $\{\phi_\ell\}_{\ell=1}^N\subset [-{\pi \over 2}, {\pi \over 2}]$ such that $\mathfrak F(\i\phi_\ell)=0$. From the mapping in \eqref{phi}, we see that for each $\ell=1,\ldots,N$,  $\phi_\ell$
 corresponds to a hyperplane $\mathcal H_\ell$ in $\HH^n$ defined by
$$ \mathcal H_\ell:=\{ (z,t)\in\HH^n:\ \phi_\ell=\phi(|z|_{\mathbb C^n},t)\}. $$
Let $$\mathcal H = \bigcup_{\ell=1}^N\mathcal H_\ell.$$
Then we see that $\{g\in \HH^n: F(g)=0\} \subset \mathcal H$, and that $\mathcal H$ has measure zero.
Consequently, the measure of the set $\{g\in \HH^n: F(g)=0\}$ is zero.

Hence, we see that the convolution kernel $K(g)$ is homogeneous of degree $-2n-2$ and  that
$$ K(g)= C_2F(g), \quad g\in \mathbb S^n, $$
which is non-zero almost everywhere on $\mathbb S^n$ (the unit sphere in $\HH^n$). Thus,
non-degeneracy condition in Theorem \ref{nondegen} holds.

\smallskip
\ \  (2c)\ \ Second order Riesz transform $X_j X_k (-\Delta_\HH)^{-1}$ on $\HH^n$, $j,k\in\{1,2,\ldots,2n\}$.

Again, by using functional calculus for $(-\Delta_\HH)^{-1}$, it is direct to see that
$$X_j X_k(-\Delta_\HH)^{-1}=  \int_0^\infty  X_j X_k e^{h \Delta_\HH} \, dh,  $$
which together with the size and smoothness estimates for the heat kernel \cite{VSC}, shows that $X_j X_k (-\Delta_\HH)^{-1}$ is a Calder\'on--Zygmund operator on $\HH^n$. Denote the kernel of $X_j X_k (-\Delta_\HH)^{-1}$ by $K_{j,k}(g)$. We now verify the
non-degeneracy condition in our Theorem \ref{nondegen}. 

In fact, this follows from similar approach as we used in (2b). Without lost of generality, we take 
$$X_j={\partial \over \partial x_j} + 2x_{n+j} {\partial\over\partial t},\ \  j<n\quad{\rm and}\quad X_k={\partial \over \partial x_k} + 2x_{n-k} {\partial\over\partial t},\ \  k>n. $$
Then based on the formula \eqref{heat kernel ph} for heat kernel, we get that for $g\not=[0,0]$,
\begin{align*}
K_{j,k}(g)=
&\, d_K^{-2n-4}(g) \Big( F_1(g)+\i F_2(g) + F_3(g) + \i F_4(g) \Big), 
\end{align*}
where 
\begin{align*}
F_1(g)&=C_3 x_jx_k\int_{\R} \big(\cosh(\lambda)\big)^{-n-2} \cosh(\lambda+\i \phi)^2\, d\lambda,\\
F_2(g)&= -  C_3  x_{n+j}x_k  \int_{\mathbb R}\big(\cosh(\lambda)\big)^{-n-2} \cosh(\lambda+\i \phi) \sinh(\lambda+\i \phi)^2\, d\lambda,\\
F_3(g)&= C_4  x_{n-k}x_j  \int_{\mathbb R}\big(\cosh(\lambda)\big)^{-n-2}\ \cosh(\lambda+\i \phi) \sinh(\lambda+\i \phi)^2 \ d\lambda,\\
F_4(g)&=  - C_4 x_{n-k}x_{n+j}  \int_{\mathbb R}\big(\cosh(\lambda)\big)^{-n-2}\   \sinh(\lambda+\i \phi)^2 \ d\lambda,
\end{align*}
with $\phi$ defined as in \eqref{phi}, $C_3:=2n(2n+2)C_1 $, $C_4:=2(2n+2)C_2$ and $C_1, C_2$ are as in \eqref{C1C2}. By resorting to the analytic continuation as in (2b) and using the isolated zero point, we get that  $K_{j,k}(g)\not=0$ a.e. $g\in \HH^n$. Thus,
non-degeneracy condition in Theorem \ref{nondegen} holds.

\bigskip

 \noindent
 {\bf Acknowledgments:}
 J. Li would like to thank Prof. Richard Rochberg for the very helpful comments on improving the paper, and thank Prof. Sundaram Thangavelu for helpful discussions. Z. Fan would like to thank Prof. Lixin Yan and Minxing Shen for helpful discussions.

M. Lacey is a 2020 Simons Fellow, his Research is supported in part by grant  from the US National Science Foundation, DMS-1949206. J. Li is supported by the Australian Research Council through the
research grant DP170101060.

\end{document}